\documentclass[12pt]{article}
\usepackage{amsmath,graphicx,amssymb,amsthm,bbm}
\usepackage{lineno}
\allowdisplaybreaks[4]

\def\cc{{\mathbb C}}

\newcommand{\sett}[1]{\left\{#1\right\}}

\newcommand{\E}{\mathbb{E}}
\newcommand{\N}{\mathcal{N}}
\newcommand{\U}{\mathcal{U}}

\def\U{\mathcal U}

\def\H{\mathcal H}

\def\L{\mathcal L}

\def\N{\mathcal N}

\def\Tr{{\rm Tr}}

\def\amslatex{$\mathcal{A}\kern-.1667em\lower.5ex\hbox{$M$}\kern-.125em\mathcal{S}$-\LaTeX}

\def\minus{\mathop{\ominus}}
\def\tensor{\mathop{\bar\otimes}}
\newtheorem{set}{set}[section]
\newtheorem{Corollary}[set]{Corollary}
\newtheorem{Definition}[set]{Definition}

\newtheorem{Lemma}[set]{Lemma}

\newtheorem{Proposition}[set]{Proposition}
\newtheorem{Remark}[set]{Remark}
\newtheorem{Theorem}[set]{Theorem}
\newcommand{\define}{\mathrel{\hbox{$\equiv$\hskip -.90em \lower .47ex \hbox{$\leftharpoondown$}}}}
\newcommand{\enifed}{\mathrel{\hbox{$\equiv$\hskip -.90em \lower .47ex \hbox{$\rightharpoondown$}}}}

\begin{document}
\title{Mixing subalgebras of finite von Neumann algebras}
\author{Jan Cameron{\small \textdagger}, Junsheng Fang {\small\textdaggerdbl}  and Kunal Mukherjee\thanks{The third author was supported in part by NSF grant DMS-0600814 during graduate studies at
Texas A\&M University.}\,\,\small{\textsection}\\
{\small \textdagger} {\small \emph{Department of Mathematics,
 Vassar College,
Poughkeepsie, NY {\rm 12604}, USA}}\\
(\small \emph{email: jacameron\@@vassar.edu})\\
{\small \textdaggerdbl} {\small\emph{ School of Mathematical Sciences,
  Dalian University of Technology,
 Dalian,  {\rm 116024}, China}}\\
(\small \emph{email: jfang@math.tamu.edu})\\
 {\small \textsection} {\small \emph{The Institute of Mathematical Sciences,
 Taramani,Chennai {\rm 600113}, India}.}\\
(\small \emph{email: kunal\@@neo.tamu.edu})
}
\date{}
 \maketitle
\begin{abstract}
 Jolissaint and Stalder introduced  definitions of mixing and weak mixing for von Neumann subalgebras of finite von Neumann algebras. In  this note, we study various algebraic and analytical properties  of  subalgebras with these mixing properties.  We prove some basic results about mixing inclusions of von Neumann algebras and establish a connection between mixing properties and normalizers of von Neumann subalgebras. The special case of mixing subalgebras arising from inclusions of  countable discrete groups finds applications to ergodic theory, in particular, a new generalization of a classical theorem of Halmos on the automorphisms of a compact abelian group.  For a finite von Neumann algebra $M$ and von Neumann subalgebras $A$, $B$ of $M$, we introduce a notion of weak mixing of $B\subseteq M$ relative to $A$. We show that weak mixing of  $B\subset M$  relative to $A$ is equivalent to the following property: if $x\in M$ and there exist a finite number of elements $x_1,\ldots,x_n\in M$ such that $Ax\subset \sum_{i=1}^nx_iB$,
 then $x\in B$.  We conclude the paper with an assortment of further examples of mixing subalgebras arising from the amalgamated free product and crossed product constructions.  \end{abstract}

 \newpage

\section{Introduction}

In~\cite{J-S}, Jolissaint and Stalder defined weak mixing and mixing for abelian von Neumann subalgebras of finite von Neumann algebras.  These properties arose as natural extensions of corresponding notions in ergodic theory in the following sense: If $\sigma$ is a measure preserving action of a countable discrete abelian group $\Gamma_0$ on a finite measure space $(X,\mu)$, then the action is (weakly) mixing in the sense of~\cite{BR} if and only if the abelian von Neumann subalgebra $L(\Gamma_0)$ is (weakly) mixing in the crossed product finite von Neumann algebra $L^\infty(X,\mu)\rtimes \Gamma_0$.

 In this note, we extend the definitions of weak mixing and mixing to general von Neumann subalgebras of finite von Neumann algebras, and study various algebraic and analytical properties of these subalgebras.  In a forthcoming note, the authors will specialize to the study of mixing properties of maximal abelian von Neumann subalgebras. If $B$ is a von Neumann subalgebra of a finite von Neumann algebra $M$, and $\E_B$ denotes the usual trace-preserving conditional expectation onto $B$, we call $B$ a \emph{weakly mixing} subalgebra of $M$ if there exists a sequence of unitary operators $\{u_n\}$ in $B$  such that

\[ \lim_{n\rightarrow\infty}\|\E_B(xu_ny)-\E_B(x)u_n\E_B(y)\|_2=0,\quad \forall x,y\in M.\]
We call $B$ a \emph{mixing} subalgebra of $M$ if the above limit is satisfied for all elements $x,y$ in $M$ and all sequences of unitary operators  $\{u_n\}$  in $B$ such that $\lim_{n\rightarrow\infty}u_n=0$ in the weak operator topology.  When $B$ is an abelian algebra, our definition of weak mixing is precisely the \emph{weak asymptotic homomorphism property} introduced by Robertson, Sinclair and Smith~\cite{RSS}. Although our definitions of weak mixing and mixing  are slightly different from those of Jolissaint and Stalder, our definitions coincide with theirs in the setting of the action of  a countable discrete group on  a probability space.  Using arguments similar to those in the proofs of Proposition 2.2 and Proposition 3.6 of~\cite{J-S}, one can show

\begin{Proposition}If $\sigma$ is a measure preserving action of a countable discrete  group $\Gamma_0$ on a finite measure space $(X,\mu)$, then the action is {\rm (}weakly{\rm )} mixing in the sense of{\rm~\cite{BR}} if and only if the von Neumann subalgebra $L(\Gamma_0)$ is {\rm (}weakly{\rm )} mixing in the crossed product finite von Neumann algebra $L^\infty(X,\mu)\rtimes \Gamma_0$.
\end{Proposition}

For an inclusion of finite von Neumann algebras $B \subseteq M$,  we call a unitary operator $u \in M$ a  \emph{normalizer} of $B$ in $M$ if $uBu^*=B$~\cite{Dix1}.  Clearly, every unitary in $B$ satisfies this condition; the subalgebra $B$ is said to be \emph{singular} in $M$ if the only normalizers of $B$ in $M$ are elements of $B$. There is a close relationship between the concepts of weak mixing and singularity.  Sinclair and Smith~\cite{S-S} noted one connection in proving that weakly mixing von Neumann subalgebras are singular in their containing algebras.  The converse was proved by Sinclair, Smith, White and Wiggins~\cite{SSWW} under the assumption that the subalgebra is also masa (maximal abelian self-adjoint subalgebra) in the ambient von Neumann algebra.  In other words, the measure preserving action of a countable discrete abelian group $\Gamma_0$ on a finite measure space $(X,\mu)$ is weakly mixing if and only if the associated von Neumann algebra $L(\Gamma_0)$ is singular in $L^\infty(X,\mu)\rtimes \Gamma_0$. This provides an operator algebraic characterization of weak mixing in the abelian setting, which is the main motivation for the study undertaken here.  In contrast to the abelian case, Grossman and Wiggins~\cite{G-W} showed that for general finite von Neumann algebras, weakly mixing is not equivalent to singularity, so techniques beyond those known for singular subalgebras are required.  In what follows, we develop basic theory for mixing properties of general subalgebras of finite von Neumann algebras.  This leads to a number of new observations about mixing properties of subalgebras and group actions, a characterization of weakly mixing subalgebras in terms of their finite bimodules, and a variety of new examples of inclusions of von Neumann algebras satisfying mixing conditions.  The paper is organized as follows.

Section \ref{section:unitary} contains some preliminary technical results.  We show that if $B$ is a diffuse finite von Neumann algebra, then $B^\omega\ominus B=\{x\in B^\omega:\, \tau_\omega(x^*b)=0,\,\forall b\in B\}$ is the weak operator closure of the linear span of unitary operators in $B^\omega\minus B$, where $B^\omega$ is the ultra-power algebra of $B$. This  result plays an important role in the subsequent sections.

In Section \ref{section:mixing}, we prove that if $B$ is a mixing von Neumann subalgebra of a finite von Neumann algebra $M$, one has
\[
 \lim_{n\rightarrow\infty}\|\E_B(xb_ny)-\E_B(x)b_n\E_B(y)\|_2=0,\quad\forall x,y\in M,
\]
if $\{b_n\}$ is a bounded sequence of operators in $B$ such that $\lim_{n\rightarrow\infty}b_n=0$ in the weak operator topology. As applications, we show that if $B$ is  mixing in $M$, $k$ is a positive integer, and $e\in B$ is a projection, then $M_k(\cc)\otimes B$ is mixing in $M_k(\cc)\otimes M$  and $eBe$ is  mixing in $eMe$. We also show that, in contrast to weakly mixing masas, one cannot distinguish mixing masas by the presence or absence of centralizing sequences in the masa for the containing II$_1$ factor.

Section \ref{section:group} concerns the special case of inclusions of group von Neumann algebras.  We extend some results of~\cite{J-S} for abelian subgroups to the case of a general inclusion of countable, discrete groups $\Gamma_0 \subset \Gamma$ in showing that $L(\Gamma_0)$ is mixing in $L(\Gamma)$ if and only if $g\Gamma_0 g^{-1}\cap \Gamma_0$ is a finite group for every $g\in \Gamma\setminus \Gamma_0$.   These two conditions are seen to be equivalent the property that for every diffuse von Neumann subalgebra $A$ of $B$ and every $y\in M$, $yAy^*\subset B$ implies $y\in B$. Some applications to ergodic theory are given.  In particular, Theorem 
\ref{T:group algebra and stongly mixing} generalizes results of Kitchens and Schmidt \cite{KS} and Halmos \cite{Ha}.

In Section \ref{section:weak mixing}, we introduce and study the concept of relative weak mixing for a triple of finite von Neumann algebras, and obtain several characterizations of weakly mixing triples.  It turns out that relative weak mixing of one subalgebra with respect to another is closely related to the bimodule structure between the two subalgebras.  In particular, we show  that $B \subset M$ is weakly mixing relative to $A$ ($A$ is not necessarily contained in $B$) if and only the following property holds: if $x\in M$ satisfies $Ax\subset \sum_{i=1}^nx_iB$ for a finite number of elements $x_1,\ldots,x_n$ in $M$, then $x\in B$. 

The results in Section \ref{section:subalgebras} show that mixing von Neumann subalgebras have hereditary properties which are notably different from those of general singular subalgebras.  We also consider the relationship between mixing and normalizers; in particular, we show that subalgebras of mixing algebras inherit a strong singularity property from the containing algebra.  Finally, we provide an assortment of new examples of  mixing von Neumann subalgebras which arise from the amalgamated free product and crossed product constructions.

We collect here some basic facts about finite von Neumann algebras, which will be used in the sequel. Throughout this paper, $M$ is a finite von Neumann algebra with a given faithful normal trace $\tau$. Denote by $L^2(M)=L^2(M,\tau)$  the Hilbert space obtained by the GNS-construction of $M$ with respect to $\tau$.
The image of $x\in M$ via the GNS-construction is denoted by $\hat{x}$, and the image of a subset $L$ of $M$ is denoted by $\hat{L}$. The trace norm of $x\in M$ is defined by $\|x\|_{2}=\|x\|_{2,\tau}=\tau(x^*x)^{1/2}$.
 Suppose that $B$ is a von Neumann subalgebra of $M$. Then there exists a unique faithful normal conditional expectation $\E_B$ from $M$ onto $B$ preserving $\tau$. Let $e_B$ be the projection of $L^2(N)$ onto $L^2(B)$. Then the von Neumann algebra $\langle M,e_B\rangle$ generated by $M$ and $e_B$ is called \emph{the basic construction} of $M$, which plays a crucial role in the study of von Neumann subalgebras of finite von Neumann algebras. There is a unique faithful tracial weight ${\rm Tr}$ on $\langle M,e_B\rangle$ such that
\[
 {\rm Tr}(xe_By)=\tau(xy),\quad\forall x,y\in M.
\]
For $\xi\in L^2(\langle M,e_B\rangle, {\rm Tr})$, define $\|\xi\|_{2,\Tr}=\Tr(\xi^*\xi)^{1/2}$.  For more details of the basic construction, we refer to~\cite{Chr, Jones, P-P, S-S2}. For a detailed account of finite von Neumann algebras and the theory of masas, we refer the reader to~\cite{S-S2}.

\noindent{\bf Acknowledgements}: The authors thank Ken Dykema, David Kerr, Roger Smith,  and Stuart White for valuable discussions throughout the completion of this work.

\section{Unitary operators in $M\ominus B$}  \label{section:unitary}

Let $M$ be a finite von Neumann algebra with a faithful normal
trace $\tau$, and let $B$ be a von Neumann subalgebra of $M$. We denote by
$M \ominus B$ the orthogonal complement of $B$ in $M$ with respect
to the standard inner product on $M$, that is, 
\[M\ominus B=\{x\in M:\tau(x^*b)=0\,\text{for all}\, b\in B\}.\] 
Then $x\in M\ominus B$ if and only if $\E_B(x)=0$, where $\E_B$ is the
trace-preserving conditional expectation of $M$ onto $B$. Note that if $x\in
M\ominus B$, then $\tau(x)=\tau(\E_B(x))=0$, so the unique positive
element in $M\ominus B$ is 0. On the other hand, it is easy to see
that $M\ominus B$ is the linear span of self-adjoint elements in
$M\ominus B.$

In the following section, we will use the fact that a bounded sequence $(b_n)$ in a finite von Neumann algebra $B$ converges to 0 in the weak operator topology if and only if it defines an element of the ultrapower $B^\omega$ which is orthogonal to $B$ in the above sense.  A key step in the proof of Theorem \ref{L:mixing to arbitary sequences} will then be to approximate an arbitrary $z \in B^\omega \ominus B$ by linear combinations of unitary operators in $B^\omega \ominus B$.  That such an approximation is possible is the main technical result of this section.  

 When  $B \subset M$ comes from an inclusion of countable discrete groups, there is an obvious dense linear subspace of $M \ominus B$:  if $G$ is a subgroup of a discrete group
$\Gamma$, then $\L(\Gamma)\ominus \L(G)$ is the weak
closure of the linear span of unitary operators corresponding to elements in
$\Gamma\setminus G$.  Although in the case of a general inclusion $B \subseteq M$, such a canonical set of unitaries is not available, we nevertheless obtain a partial answer to the following question:  If $B$ is a subalgebra of a diffuse finite von Neumann algebra $M$ such that $eMe \neq eBe$ for every nonzero projection $e \in B$, is $M \ominus B$ the weak closure of the linear span of unitaries in $M \ominus B$ ?

The assumption that $eMe\neq eBe$ for every nonzero projection $e\in B$ is
necessary, as is the assumption that $M$ is diffuse.
For instance, if $M=\cc\oplus \cc$ and $B=\cc$ and $\tau(1\oplus
0)\neq \tau(0\oplus 1)$, then there are no unitary operators in
$M\ominus B$.

Let $(M)_1$ be the closed unit ball of $M$, and let $\Lambda=\{x\in
(M)_1: x=x^*, \E_B(x)=0\}$. Then $\Lambda$ is a convex set which is
closed, hence also compact, in the weak operator topology.  By the Krein-Milman
Theorem, $\Lambda$ is the weak operator closure of the convex hull of its extreme points.
 Thus, we need only  characterize the extreme points of
$\Lambda$.

\begin{Lemma}\label{L:extreme points of gamma} Suppose that for every
nonzero projection $p\in M$, there exists a nonzero element $x_p\in
pMp$ satisfying $\E_B(x_p)=0$.
 Then the extreme points of $\Lambda$ are
\[\left\{2e-1:\,e\in M\,\text{and}\,\, \E_B(e)=\frac{1}{2}\right\}.
\]
\end{Lemma}
\begin{proof} Suppose $a\in \Lambda$ is an extreme point of
$\Lambda$  and $a\neq 2e-1$. By the spectral decomposition theorem,
there exists an $\epsilon>0$ and a nonzero spectral projection $e$ of
$a$ such that
\[(-1+\epsilon)e\leq ae \leq (1-\epsilon)e.
\] By assumption, there is a nonzero self-adjoint
element $x\in eMe$ such that
 $\E_B(x)=0$. By multiplying a
scalar, we may assume that $-\epsilon e\leq x\leq \epsilon e$.
Then $x+a,x-a\in \Lambda$ and
$x=\frac{1}{2}(x+a)+\frac{1}{2}(x-a)$. So $x$ is not an extreme
point of $\Lambda$, contradicting our assumption. Therefore,
$x=2e-1$ for some projection $e\in M$. Since $\E_B(x)=0$,
$\E_B(e)=\frac{1}{2}$. Note that $2e-1$ is an extreme point of
$(M)_1$.  This completes the proof.
\end{proof}

The following example shows that the assumptions of the above Lemma are
essential.

\noindent{\bf Example:}\, In the inclusion  $\cc
\subseteq M_3(\cc)$, there is no projection $e\in M_3(\cc)$
satisfing $\tau(e)=\frac{1}{2}$. In this case, the
partial isometry
\[\left(\begin{array}{ccc}
1&0&0\\
0&0&0\\
0&0&-1
\end{array}
 \right)
\] is an extreme point of $\Lambda$.

\begin{Corollary}\label{L:M minus C} Let $M$ be a diffuse finite von Neumann
algebra with a faithful normal trace $\tau$. Then $M\ominus \cc 1$
is the weak operator closure of the linear span of self-adjoint
unitary operators in $M\ominus \cc 1$.
\end{Corollary}
\begin{proof} For every nonzero projection $p\in M$, $pMp$ is diffuse and hence
$pMp\neq \cc p$. So there is an operator $x_p\in pMp$ with
$\tau(x_p)=0$. By Lemma~\ref{L:extreme points of gamma}, $M\ominus
\cc 1$ is the weak operator closure of linear span of self-adjoint
unitary operators in $M\ominus \cc 1$.
\end{proof}

\begin{Lemma}\label{L: haar unitary operators} Suppose $B$ is a
diffuse finite von Neumann algebra with a faithful normal trace
$\tau$. For $\epsilon>0$ and $x_1,\ldots,x_n\in B$, there exists a
Haar unitary operator $u\in B$ such that
\[|\tau(x_iu^*)|<\epsilon,\quad  1\leq i\leq n.
\]
\end{Lemma}
\begin{proof} Since $B$ is diffuse, $B$ contains a Haar unitary
operator $v$. Note that $v^n\rightarrow 0$ in the weak operator
topology. So there exists an $N$ such that
\[|\tau(x_i(v^N)^*)|<\epsilon,\quad  1\leq i\leq n.
\]
Let $u=v^N$. Then $u$ is a Haar unitary operator and the lemma
follows.
\end{proof}

Let $B$
be a separable diffuse  von Neumann algebra with a faithful normal
trace $\tau$, and let $B^\omega$ be the ultrapower algebra of $B$
with a faithful normal trace $\tau_\omega$ (see~\cite{Sa}). Let $(B^\omega\ominus B)_1$ be the closed unit ball of $B^\omega\ominus B$.
 The following proposition is the main result of this section.
\begin{Proposition} \label{P:ultrapower algebra minus B}Suppose $B$ is a separable diffuse finite von Neumann algebra with a faithful normal trace $\tau$. Then  $(B^\omega\ominus
B)_1$ is  the trace norm closure of the convex hull of
self-adjoint unitary operators in $B^\omega\ominus B$.
\end{Proposition}
\begin{proof} We claim that for every nonzero projection $p\in
B^\omega$,  there exists a nonzero element $x_p$ in $pB^\omega p$ such
that $\E_{B}(x_p)=0$, where $\E_{B}$ is the conditional
expectation of $B^\omega$ onto $B$ preserving $\tau_\omega$. Let
$p=(p_n)$, where $p_n\in B$ is a projection with
$\tau(p_n)=\tau_\omega(p)>0$. Since $B$ is separable, there is a
dense sequence $\{y_n\}$ of $B$ in the trace norm. We may assume
that $y_1=1$. By Lemma~\ref{L: haar unitary operators}, for any finite subset
$\sett{y_1,\ldots,y_n}$ of the dense sequence, there is a Haar unitary operator $u_n\in
p_nBp_n$ such that
\[|\tau(p_ny_ip_nu_n^*)|<\frac{1}{n},\quad \forall 1\leq i\leq n.
\]
Let $x_p=(u_n)$. Then $\|x_p\|_2^2=\lim_{n\rightarrow\omega}
\|u_n\|_2^2=\tau(p)>0$. Hence, $x_p\neq 0$ and $x_p\in pB^\omega
p$. Note that for each $y_k$,
\[\tau_\omega(y_k(x_p)^*)=\tau_\omega((py_kp)(x_p)^*)=\lim_{n\rightarrow\omega}\tau(p_ny_kp_nu_n^*)=0.
\]
Since $\{y_k\}$ is dense in $B$ in the trace norm topology,
$\tau_\omega(y(x_p)^*)=0$ for all $y\in B$. This implies
$\E_{B}(x_p)=0$. By Lemma~\ref{L:extreme points of gamma}, $(B^\omega\ominus B)_1$ is the weak operator closure of  the convex hull of
self-adjoint unitary operators in $B^\omega\ominus B$. Note that $(B^\omega\ominus B)_1$ is a convex set, so its weak operator closure coincides with its closure in the strong operator and trace norm topologies.  This proves the result.
\end{proof}

\begin{Corollary} \label{T:orthocomplement}
 Suppose $B$ is a separable diffuse finite von Neumann algebra with a faithful normal trace $\tau$. Then  $B^\omega\ominus B$ is  the weak operator closure of the linear span of
self-adjoint unitary operators in $B^\omega\ominus B$.
\end{Corollary}

Using a similar approach, we can also prove the following result.

\begin{Proposition} \label{P:span}
If $M$ is a separable type ${\rm II}_1$ factor and $B$ is an abelian von Neumann subalgebra of $M$, then $M\ominus B$ is the weak operator closure of the linear span of unitary operators in $M\ominus B$. 
\end{Proposition}

It is not clear whether Proposition \ref{P:span} holds for nonabelian subalgebras.  We are unable, for instance, to establish the conclusion of the result when $B$ is a hyperfinite subfactor of a nonhyperfinite type ${\rm II}_1$ factor $M$, e.g. $L\mathbb{F}_2$.

\section{Mixing von Neumann subalgebras}  \label{section:mixing}

Let $M$ be a finite von Neumann algebra with a faithful normal
trace $\tau$, and let $B$ be a von Neumann subalgebra of $M$.

\begin{Definition}\label{D:mixing}\emph{An algebra $B$ is a
\emph{mixing von Neumann subalgebra} of $M$ if
\[\lim_{n\rightarrow \infty}\|\E_B(xu_ny)-\E_B(x)u_n\E_B(y)\|_2=0
\] holds for all $x,y\in M$ and every sequence of unitary operators $\{u_n\}$  in $B$ such that
$\displaystyle\lim_{n\rightarrow \infty}u_n=0$ in the weak
operator topology. If $B$ is a mixing von Neumann subalgebra of $M$, then we say  $B\subseteq M$ a \emph{mixing inclusion} of finite von Neumann algebras.}
\end{Definition}

It is easy to see that $B$ is a mixing von Neumann subalgebra
of $M$ if and only if for all elements $x,y$ in $M$ with
$\E_B(x)=\E_B(y)=0$, one has
\[\lim_{n\rightarrow \infty}\|\E_B(xu_ny)\|_2=0
\] if $\{u_n\}$ is a sequence of unitary operators in $B$ such that
$\displaystyle\lim_{n\rightarrow \infty}u_n=0$ in the weak
operator topology.

\begin{Remark}\label{R:mixing}\emph{ By the Kaplansky density theorem,  we may assume that $x$ and $y$ are in a
subset $F$ of $M$ such that $M$ is the von Neumann algebra
generated by $F$ in
Definition~\ref{D:mixing}.}
\end{Remark}

The following theorem provides a useful equivalent condition for mixing inclusions of finite von Neumann algebras.
\begin{Theorem}\label{L:mixing to arbitary sequences}
 If $B$ is a mixing von Neumann subalgebra
of $M$ and $x,y\in M$ with $\E_B(x)=\E_B(y)=0$, then
\[\lim_{n\rightarrow \infty}\|\E_B(xb_ny)\|_2=0
\] if $\{b_n\}$ is a bounded sequence of operators in $B$ such that
$\displaystyle\lim_{n\rightarrow \infty}b_n=0$ in the weak
operator topology.
\end{Theorem}
\begin{proof}
Let $\omega$ be a free ultrafilter of the set of natural numbers and let $M^\omega$ be the ultrapower algebra of $M$. Then $M^\omega$ is a finite von Neumann algebra with a faithful normal trace $\tau_\omega$.
We can identify $B^\omega$ with a von Neumann subalgebra of $M^\omega$ in the natural way.  Every bounded sequence $(b_n)$ of $B$ defines an element $z$ in $B^\omega$. We may assume that $\|z\|\leq 1$. It is easy to see that $\lim_{n\rightarrow \omega}b_n=0$ in the weak operator topology if and only if
\[
 \tau_\omega(zb)=0,\quad \forall b\in B.
\]
Recall that
$M\ominus B=\{x\in M:\tau(x^*b)=0\,\text{for all}\, b\in B\}$.  It is easy to see Definition~\ref{D:mixing} is equivalent to  the following:  For any  elements $x,y$ in
$M\ominus B$, and $u\in B^\omega\ominus B$, one has $\E_{B^\omega}(xuy)=0$.
By Proposition~\ref{P:ultrapower algebra minus B}, $(B^\omega\ominus B)_1$ is the trace norm closure of the convex hull of unitary operators in $B^\omega\ominus B$. Let $\epsilon>0$. Then there exist unitary operators $u_1,\ldots, u_n$ in $B^\omega\ominus B$ and positive numbers $\alpha_1,\ldots, \alpha_n$ with $\alpha_1+\cdots+\alpha_n=1$ such that
\[
 \left\|z-\sum_{k=1}^n\alpha_ku_k\right\|_{2,\tau_\omega}<\epsilon.
\]
For any  elements $x$ and $y$ in
$M\ominus B$,
\[
 \|\E_{B^\omega}(xzy)\|_{2,\tau_\omega}=\left\|\E_{B^\omega}\left(x\left(z-\sum_{k=1}^n\alpha_ku_k\right)y\right)\right\|_{2,\tau_\omega}
\leq \left\|x\left(z-\sum_{k=1}^n\alpha_ku_k\right)y\right\|_{2,\tau_\omega}\]\[\leq \|x\|\cdot \left\|z-\sum_{k=1}^n\alpha_ku_k\right\|_{2,\tau_\omega}\cdot\|y\|\leq \epsilon \|x\|\|y\|.
\]
Since $\epsilon>0$ is arbitrary, $\E_{B^\omega}(xzy)=0$, which is equivalent to
\[\lim_{n\rightarrow \infty}\|\E_B(xb_ny)\|_2=0.
\]
\end{proof}

Two applications of the above theorem are the following.

\begin{Corollary}
 If $B$ is a mixing von Neumann subalgebra of $M$ and $k$ is a positive integer, then $M_k(\cc)\otimes B$ is mixing in $M_k(\cc)\otimes M$.
\end{Corollary}
\begin{proof}
 Note that $x=(x_{ij})\in (M_k(\cc)\otimes M)\ominus (M_k(\cc)\otimes B)$ if and only if $x_{ij}\in M\ominus B$ for all $1\leq i,j\leq k$, and $b_n=(b_{ij}^n)\in M_k(\cc)\otimes B$ converges to 0 in the weak operator topology if and only if $b_{ij}^n$ converges to 0 in the weak operator topology for all $1\leq i,j\leq k$. Now the corollary follows from Theorem~\ref{L:mixing to arbitary sequences}.
\end{proof}

\begin{Corollary}\label{L:cut down of strongly mixing subalgebras}
 If $B$ is a mixing von Neumann subalgebras of $M$ and $e$ is a projection of $B$, then $eBe$ is mixing in $eMe$.
\end{Corollary}
\begin{proof}
 Let $(b_n)$ be a bounded sequence of $eBe$ which converges to 0 in the weak operator topology. For $x,y\in eMe\ominus eBe$, we have $x,y\in M\ominus B$. By Theorem~\ref{L:mixing to arbitary sequences},
\[\lim_{n\rightarrow \infty}\|\E_{eBe}(xb_ny)\|_2=\lim_{n\rightarrow \infty}\|\E_{B}(xb_ny)\|_2=0.
\]\end{proof}
It is well-known that the presence of centralizing sequences in a masa for its containing II$_1$ factor is a  conjugacy invariant for the masa.  More generally, it is possible to build non-conjugate masas of a II$_1$ factor by controlling the existence of centralizing sequences in various cutdowns of each masa.  Sinclair and White~\cite{S-W} developed this technique to produce uncountably many non-conjugate weakly mixing masas in the hyperfinite $\rm{II}_1$ factor with the same Puk\'{a}nszky invariant. The final result of this section implies that, in contrast to the larger class of weakly mixing masas, there is no hope of distinguishing mixing masas along these lines.  Following the notation of~\cite{S-W}, for a von Neumann subalgebra $B$ of a ${\rm II}_1$ factor $M$, we denote  by $\Gamma(B)$ the maximal trace of a projection $e\in B$ for which $eBe$ contains a non-trivial centralizing sequences for $eMe$.

\begin{Proposition}
 If $B$ is a  mixing subalgebra of a type ${\rm II}_1$ factor $M$ and $eBe\neq eMe$ for each nonzero projection $e\in B$,  then $\Gamma(B)=0$.
\end{Proposition}
\begin{proof}
 By Corollary~\ref{L:cut down of strongly mixing subalgebras}, we need only show that there is no nontrivial central sequence $\{b_n\}$ in $B$ of $M$. Suppose $\{b_n\}\subset B$ is a  central sequence  of $M$.  We may assume that $\tau(b_n)=0$ for each $n$. Suppose $\lim_{n\rightarrow \omega}b_n=z\in B$ in the weak operator topology, then for all $x \in M,$
\[zx=\lim_{n\rightarrow\omega}b_nx=\lim_{n\rightarrow\omega}xb_n=xz.\]

Since $M$ is a type ${\rm II}_1$ factor, $z=\tau(z)1=0$. Hence $\lim_{n\rightarrow \omega}b_n=0$ in the weak operator topology.
Choose a nonzero element $x\in M$ such that $\tau(xb)=0$ for all $b\in B$. Note that
\[
 \|xb_n-b_nx\|_2^2=\|xb_n\|_2^2+\|b_nx\|_2^2-2Re\tau(b_n^*x^*b_nx)
\]
\[\geq \tau(b_n^*x^*xb_n)-2Re\tau(b_n^*\E_B(x^*b_nx))\]\[=\tau(x^*xb_nb_n^*)-2Re\tau(b_n^*\E_B(x^*b_nx)).\]
Since $\{b_n\}$ is a central sequence of $M$, $\{b_nb_n^*\}$ is also a central sequence of $M$. The uniqueness of the trace on $M$ implies that
\[
\lim_{n\rightarrow \omega}\tau(x^*xb_nb_n^*) =\lim_{n\rightarrow\omega}\tau(x^*x)\tau(b_nb_n^*)=\lim_{n\rightarrow\omega}\|x\|_2^2 \cdot \|b_n\|_2^2.
\]
By Theorem~\ref{L:mixing to arbitary sequences},
\[
0=\lim_{n\rightarrow \infty}\|xb_n-b_nx\|_2\geq \|x\|_2\lim_{n\rightarrow \infty}\|b_n\|_2,
\] which implies that $\lim_{n\rightarrow\omega}\|b_n\|_2=0$. This completes the proof.
\end{proof}

\begin{Corollary}
 If $B$  is a mixing masa of a type ${\rm II}_1$ factor $M$, then $\Gamma(B)=0$.
\end{Corollary}

\section{Mixing inclusions of group von Neumann algebras}  \label{section:group}
In this section, we apply our operator-algebraic machinery to the special case
of mixing inclusions of von Neumann algebras that arise from actions of countable, discrete
groups.  This direction was taken up in~\cite{J-S}, where it was shown that, for an infinite abelian subgroup $\Gamma_0$ of a countable group $\Gamma,$ the inclusion $L(\Gamma_0) \subset L(\Gamma)$ is mixing if and only if the following condition (called (ST)) is satisfied: 

\medskip
  For every finite subset $C$ of $\Gamma \setminus \Gamma_0,$ there exists a finite exceptional set $E \subset \Gamma_0$ such that $g \gamma h \notin \Gamma_0$ for all $g_0 \in \Gamma_0 \setminus E$ and $g,h \in C$.  \medskip

Theorem \ref{T:group algebra and stongly mixing} of this section supplies a similar characterization for the case in which $\Gamma_0$ is not abelian, and also establishes a connection between the group normalizer of the subgroup $\Gamma_0$ and the ``analytic" normalizer of its associated group von Neumann algebra.  The key observation required is the following, which shows that mixing subalgebras satisfy a much stronger form of singularity.  

\begin{Theorem}\label{T:mixing} Let $B$ be a  mixing von
Neumann subalgebra of $M$, and
suppose that $A$ is a diffuse von Neumann subalgebra of  $B$. If $y \in M$ satisfies
$yAy^*\subseteq B$, then $y\in B$.
\end{Theorem}
\begin{proof}  We may assume that $A$ is a
diffuse abelian von Neumann algebra. Then $A$ is generated by a
Haar unitary operator $w$. In particular,
$\lim_{n\rightarrow\infty}w^n=0$ in the weak operator topology.
Let $x\in M$ and $\E_B(x)=0$. Then
\[
|\tau(xy)|^2\leq \|\E_{A'\cap M}(xy)\|_2^2.
\]
 Note that
\[\E_{A'\cap M}(xy)=\lim_{n\rightarrow \omega} \frac{\sum_{k=1}^n w^k(xy)(w^*)^k}{n}
\] in the weak operator topology. Hence,
\begin{eqnarray*}
  |\tau(xy)|^2 &\leq & \|\E_{A'\cap M}(xy)\|_2^2 \\
   &\leq & \lim_{n\rightarrow \omega} \left\|\frac{\sum_{k=1}^n w^k(xy)(w^*)^k}{n}\right\|_2^2 \\
   &=&\lim_{n\rightarrow \omega}\frac{1}{n^2}\sum_{i,j=1}^n \tau(w^i(xy)(w^*)^iw^j(y^*x^*)(w^*)^j) \\
   &\leq & \lim_{n\rightarrow \omega}\frac{1}{n^2}\sum_{i,j=1}^n|\tau(x(yw^{j-i}y^*)x^*(w^*)^{j-i})| \\
   &\leq & \lim_{n\rightarrow
   \omega}\frac{1}{n^2}\sum_{i,j=1}^n\|\E_B(x(yw^{j-i}y^*)x^*(w^*)^{j-i})\|_2\\
   &=& \lim_{n\rightarrow
   \omega}\frac{1}{n^2}\sum_{i,j=1}^n\|\E_B(x(yw^{j-i}y^*)x^*)\|_2.
\end{eqnarray*}
By hypothesis, $yw^{n}y^*\in B$.  Note that
$\lim_{n\rightarrow \infty}yw^{n}y^*=0$ in the weak operator
topology. By Theorem~\ref{L:mixing to arbitary sequences},
\[\lim_{n\rightarrow\infty}\|\E_B(x(yw^{n}y^*)x^*)\|_2=0.
\]
So
\[|\tau(xy)|^2\leq \lim_{n\rightarrow
   \omega}\frac{1}{n^2}\sum_{i,j=1}^n\|\E_B(x(yw^{j-i}y^*)x^*)\|_2=0.
\]
Therefore, $\tau(xy)=0$ for all $y\in M\ominus B$.
This implies that $y\in B$.
\end{proof}

\begin{Remark}\emph{In Theorem ~\ref{T:mixing}, it is not necessary that the unit of $A$ be the same as the unit of $B$. }
\end{Remark}

\begin{Theorem}\label{T:group algebra and stongly mixing} Let
$M=L(\Gamma)$ and $B=L(\Gamma_0)$. Then the following conditions
are equivalent:
\begin{enumerate}
\item $B=L(\Gamma_0)$ is mixing in $M=L(\Gamma)$;

\item $g\Gamma_0 g^{-1}\cap \Gamma_0$ is a finite group for every
$g\in \Gamma\setminus \Gamma_0$;

\item for every diffuse von Neumann subalgebra $A$ of $B$ and
every unitary operator $v\in M$, if $vAv^*\subseteq B$, then $v\in
B$;

\item for every diffuse von Neumann subalgebra $A$ of $B$ and
every  operator $y\in M$, if $yAy^*\subseteq B$, then $y\in
B$.
\end{enumerate}
\end{Theorem}
\begin{proof}
``$1\Rightarrow 4$" follows from Theorem~\ref{T:mixing} and ``$4\Rightarrow 3$'' is trivial.

``$3\Rightarrow 2$". Suppose $M=L(\Gamma)$ and $B=L(\Gamma_0)$.
Suppose for some $g\in \Gamma\setminus \Gamma_0$,
$g\Gamma_0g^{-1}\cap \Gamma_0$ is an infinite group. Let
$\Gamma_1=\Gamma_0\cap g^{-1}\Gamma_0 g=g^{-1}(g\Gamma_0g^{-1}\cap
\Gamma_0)g$. Then $\Gamma_1$ is an infinite group, and $g\Gamma_1
g^{-1}\subseteq \Gamma_0$. So $\lambda(g)
L(\Gamma_1)\lambda(g^{-1})\subseteq L(\Gamma_0)$. By the
third statement, $\lambda(g)\in L(\Gamma_0)$ and $g\in
\Gamma_0$. This is a contradiction.

 ``$2\Rightarrow 1$". First, we show that if $g_1,g_2\in \Gamma\setminus
 \Gamma_0$, then $g_1\Gamma_0 g_2\cap \Gamma_0$ is a finite set.
 Suppose $h_1,h_2\in \Gamma_0$ and $g_1h_1g_2, g_1h_2g_2 \in
 \Gamma_0$. Then
 \[g_1h_1h_2^{-1}g_1^{-1}=g_1h_1g_2(g_1h_2g_2)^{-1}\in
 \Gamma_0\cap g_1\Gamma_0g_1^{-1}.
 \] Since $\Gamma_0\cap g_1\Gamma_0g_1^{-1}$ is a finite group,
 $\{h_1h_2^{-1}:\, h_1,h_2\in \Gamma_0\,\text{and}\, g_1h_1g_2, g_1h_2g_2 \in
 \Gamma_0\}$ is a finite set. Hence, $g_1\Gamma_0 g_2\cap \Gamma_0$ is a finite
 set.

 Let $\{v_n\}$ be a sequence of unitary operators in $B$ such that
 $\displaystyle\lim_{n\rightarrow\infty}v_n=0$ in the weak
 operator topology. Write $v_n=\sum_{k=1}^\infty \alpha_{n,k}\lambda(h_k)$.
 Then for each $k$, $\lim_{n\rightarrow\infty} \alpha_{n,k}=0$.
 Suppose $g_1,g_2\in \Gamma\setminus\Gamma_0$. There exists an $N$
 such that for all $m\geq N$, $g_1h_mg_2\notin \Gamma_0$. Hence,
 \[
\|\E_B(g_1v_ng_2)\|_2=\sum_{i=1}^N\|\alpha_{n,i}\E_B(g_1\lambda(h_i)g_2)\|_2\leq
\sum_{i=1}^N|\alpha_{n,i}|\rightarrow 0
 \]
when $n\rightarrow \infty$. By Remark~\ref{R:mixing},
 $M$ is mixing relative to $B$.

\end{proof}

We now  apply Theorem~\ref{T:group algebra and stongly mixing} to the group-theoretic situation arising from a semidirect product $\Gamma=G\rtimes \Gamma_0$, where $\Gamma_0$ is an infinite group. Let $\sigma_h(g)=hgh^{-1}$ for $h\in \Gamma_0$ and $g\in G$. Then $\sigma_h$ is an automorphism of $G$. Note that $hg=hgh^{-1}h=\sigma_h(g)h$ for $h\in \Gamma_0$ and $g\in G$.

\begin{Proposition}\label{P:mixing of semidirect product}  Let $M=L(G\rtimes \Gamma_0)$ and $B=L(\Gamma_0)$. Then $B$ is mixing in $M$ if and only if for each $g\in G$, $g\neq e$,  the group
\[
 \{h\in \Gamma_0:\sigma_h(g)=g\}
\]
is finite.
\end{Proposition}
\begin{proof}
 Let $g\in G$ and $h\in \Gamma_0$.  Suppose  $h\in g\Gamma_0g^{-1}\cap \Gamma_0$. Then $ghg^{-1}\in \Gamma_0$. Note that $ghg^{-1}=hh^{-1}ghg^{-1}=h(\sigma_{h^{-1}}(g)g^{-1})$. So $ghg^{-1}\in \Gamma_0$ implies that $\sigma_{h^{-1}}(g)g^{-1}\in \Gamma_0\cap G=\{e\}$, i.e., $\sigma_{h^{-1}}(g)=g$ and hence $\sigma_h(g)=g$.
Conversely, suppose $\sigma_h(g)=g$. Then $\sigma_{h^{-1}}(g)=g$ and hence $ghg^{-1}=h\sigma_{h^{-1}}(g)g^{-1}=h\in \Gamma_0\cap g\Gamma_0g^{-1}$. This proves
\[
 \{h\in \Gamma_0: \sigma_h(g)=g\}=\{h\in \Gamma_0: h\in g\Gamma_0g^{-1}\cap \Gamma_0\}
\]

Suppose $B$ is mixing in $M$. By 2 of Theorem~\ref{T:group algebra and stongly mixing}, $g\Gamma_0g^{-1}\cap \Gamma_0$ is a finite group for every $g\in G$ with $g\neq e$. So the group $\{h\in H:\sigma_h(g)=g\}$ is finite.  Conversely, suppose for each $g\in G$, $g\neq e$,  the group
$
 \{h\in \Gamma_0:\sigma_h(g)=g\}
$
is finite, which implies that  $g\Gamma_0g^{-1}\cap \Gamma_0$ is finite.    A group element of $\Gamma\setminus \Gamma_0$ can be written as $gh$, $g\in G$, $g\neq e$, $h\in \Gamma_0$.  Note that
\[
 gh\Gamma_0h^{-1}g^{-1}\cap \Gamma_0=g\Gamma_0 g^{-1}\cap \Gamma_0
\]
is finite. So $B$ is mixing in $M$ by 2 of Theorem~\ref{T:group algebra and stongly mixing}.
\end{proof}

Recall that the action $\sigma$ of a group $H$ on a finite von Neumann algebra $N$ is called \emph{ergodic} if $\sigma_h(x)=x$ for all $h\in H$ implies that $x=\lambda 1$. The following result extends Theorem 2.4 of~\cite{KS} to the noncommutative setting.
\begin{Corollary}
 Let $M=L(G\rtimes \Gamma_0)$ and $B=L(\Gamma_0)$. Suppose $\Gamma_0$ is a finitely generated, infinite, abelian group or $\Gamma_0$ is a torsion free group. Then $B$ is mixing in $M$ if and only if every element $h\in \Gamma_0$ of infinite order is ergodic on $L(G)$.
\end{Corollary}
\begin{proof}
 If $B$ is mixing in $M$, then clearly every element $h\in \Gamma_0$ of infinite order is ergodic on $L(G)$.  Now suppose every element $h\in \Gamma_0$ of infinite order is ergodic on $L(G)$.  If $B$ is not mixing in $M$, then there is a $g\in G$, $g\neq e$, such that $\{h\in \Gamma_0:\sigma_h(g)=g\}$ is an infinite group. Under the above hypotheses on $\Gamma_0$, there exits an element $h_0$ of infinite order such that $\sigma_{h_0}(g)=g$. This implies that the action of $h_0$ on $L(G)$ is not ergodic, which is a contradiction.
\end{proof}

\begin{Corollary}\label{C:a generalization of Halmos}
 Let $M=L(G\rtimes \mathbb{Z})$ and $B=L(\mathbb{Z})$. Then the following conditions are equivalent:
\begin{enumerate}
 \item the action of $\mathbb{Z}$ on $L(G)$ is mixing, i.e.,  $B$ is mixing in $M$;
\item the action of $\mathbb{Z}$ on $L(G)$ is weakly mixing, i.e.,  $B$ is weakly mixing in $M$;
\item the action of $\mathbb{Z}$ on $L(G)$ is ergodic;
\item for every $g\in G$, $g\neq e$, the orbit $\{\sigma_h(g)\}$ is infinite;
\item for every $g\in G$, $g\neq e$, $\{h\in \mathbb{Z}: \sigma_h(g)=g\}=\{e\}$.
\end{enumerate}
\end{Corollary}
\begin{proof}
Let $\gamma$ be a generator of $\mathbb{Z}$.  Clearly ``$1\Rightarrow 2\Rightarrow 3$''.

``$3\Rightarrow 4$''. Suppose $\sigma_{\gamma^n}(g)=g$ and $n$ is the minimal positive integer satisfies this condition. Let $x=L_g+L_{\sigma_{\gamma}(g)}+\ldots+L_{\sigma_{\gamma^{n-1}}(g)}$. Then $x\in L(G)$, $x\neq \lambda 1$, and $\sigma_h(x)=x$ for all $h\in \mathbb{Z}$. This implies that the action of $\mathbb{Z}$ on $L(G)$ is not ergodic.

``$4\Rightarrow 5$''. Suppose $\sigma_{\gamma^n}(g)=g$ for some positive integer $n$. Then the orbit $\{\sigma_h(g)\}$ has at most $n$ elements.

``$5\Rightarrow 1$'' follows from Proposition~\ref{P:mixing of semidirect product}.
\end{proof}

A special case of Corollary~\ref{C:a generalization of Halmos} implies the following classical result of Halmos~\cite{Ha}.
\begin{Corollary}[Halmos's Theorem]
 Let $X$ be a compact abelian group, and $T:X\rightarrow X$ a continuous automorphism. Then $T$ is mixing if and only if $T$ is ergodic.
\end{Corollary}
\begin{proof}
 By the Pontryagin duality theorem, the dual group $G$ of $X$ is a discrete abelian group. Furthermore, there is an induced action of $\mathbb{Z}$ on $G$, and the action is unitarily conjugate to the action of $T$ on $X$. Now the corollary follows from Corollary~\ref{C:a generalization of Halmos}.
\end{proof}

\section{Relative weak mixing}  \label{section:weak mixing}

Suppose $M$ is a finite von Neumann algebra with a faithful normal trace $\tau$, and $A$, $B$ are von Neumann subalgebras of $M$.  We say $B\subset M$ is \emph{weakly mixing relative to $A$} if there exits a sequence of unitary operators $u_n\in A$ such that
\[
\lim_{n\rightarrow\infty}\|\E_B(xu_ny)-\E_B(x)u_n\E_B(y)\|_2=0,\quad \forall x,y\in M.
\]
So $B$ is weakly mixing in $M$ if and only if $B\subset M$ is weakly mixing relative to $B$. Since every diffuse von Neumann algebra contains a sequence of unitary operators converging to 0 in the weak operator topology, $B$ is mixing in $M$ implies that $B\subset M$ is weakly mixing relative to $A$ for all diffuse von Neumann subalgebras $A$ of $B$.

It is easy to see that $B\subset M$ is weakly mixing relative to $A$ if and only if there exits a sequence of unitary operators $u_n\in A$ such that for all elements $x,y$ in $M$ with
$\E_B(x)=\E_B(y)=0$, one has
\[\lim_{n\rightarrow \infty}\|\E_B(xu_ny)\|_2=0.
\]

The main result of this section is the following, which is inspired by~\cite{Po3}.

\begin{Theorem}\label{T:main result}
  Let $M$ be a finite von Neumann algebra with a faithful normal trace $\tau$, and let $A$, $B$ be  von Neumann subalgebras of $M$. Then the following conditions are equivalent:
\begin{enumerate}
 \item $B\subset M$ is weakly mixing relative to $A$, i.e., there exists a sequence of unitary operators $\sett{u_k}$ in $A$ such that
\[
 \lim_{k\rightarrow\infty}\|\E_B(xu_ky)\|_2=0,\quad\forall x,y\in N\ominus B;
\]
\item if $z\in A'\cap \langle M,e_B\rangle$ satisfies $\Tr(z^*z)<\infty$, then $e_Bze_B=z$;
\item if $p\in A'\cap \langle M,e_B\rangle$ satisfies $\Tr(p)<\infty$, then $e_Bpe_B=p$;
\item if $x\in M$ satisfies $Ax\subset \sum_{i=1}^nx_i B$ for a finite number of elements $x_1,\ldots,x_n\in M$, then $x\in B$.
\end{enumerate}
\end{Theorem}

Before we prove Theorem~\ref{T:main result}, we state some corollaries of the theorem.

\begin{Corollary}\label{C:weak mixing}
 Let $M$ be a finite von Neumann algebra with a faithful normal trace $\tau$, and let $B$ be a  von Neumann subalgebra of $M$.  Then the following conditions are equivalent:
\begin{enumerate}
 \item $B$ is a weakly mixing von Neumann subalgebra of $M$;
\item if $x\in M$ satisfies $Bx\subset \sum_{i=1}^nx_i B$ for a finite number of elements  $x_1,\ldots,x_n\in M$, then $x\in B$.
\end{enumerate}
\end{Corollary}

The following corollary gives an operator algebraic characterization of weak mixing actions of countable discrete groups.

\begin{Corollary}
 If $\sigma$ is a measure preserving action of a countable discrete  group $\Gamma_0$ on a finite measure space $(X,\mu)$, then weak mixing of $\sigma$ is equivalent to the following property: if $x\in L^\infty(X,\mu)\rtimes \Gamma_0$ and $xL(\Gamma_0)\subset \sum_{i=1}^n x_iL(\Gamma_0)$ for a finite number of elements  $x_1,\ldots,x_n$ in $L^\infty(X,\mu)\rtimes \Gamma_0$, then $x\in L(\Gamma_0)$.

\end{Corollary}

\begin{Corollary}\label{C:strong mixing}
 Let $M$ be a finite von Neumann algebra with a faithful normal trace $\tau$, and let $B$ be a mixing von Neumann subalgebra of $M$.  If $A\subset B$ is a diffuse von Neumann subalgebra and $x\in M$ satisfies $Ax\subset \sum_{i=1}^nx_i B$ for a finite number of elements $x_1,\ldots,x_n\in M$, then $x\in B$.
\end{Corollary}

To prove Theorem~\ref{T:main result}, we need the following lemmas.

 \begin{Lemma}\label{L:approximate projections} Let $p\in \langle M,e_B \rangle$ be a finite projection, $p\leq 1-e_B$, and
$\epsilon>0$. Then there exist
$x_1,\ldots,x_n\in M\ominus B$ such that $\E_B(x_j^*x_i)=\delta_{ij}f_i$,
where $f_i$ is a projection in $B$, and
\[\left\|p-\sum_{i=1}^n x_ie_B x_i^*\right\|_{2,\Tr}<\epsilon.\]
\end{Lemma}
\begin{proof}
 Let $q=e_B+p$. Then $q$ is a finite projection in $\langle M,e_B\rangle$. By Lemma 1.8 of~\cite{Po1}, there are $x_0,x_1,\ldots, x_n\in M$, $x_0=1$, such that $\E_B(x_j^*x_i)=\delta_{ij}f_i$ for $0\leq i,j\leq n$ and
\[
 \left\|q-\sum_{i=0}^n x_ie_B x_i^*\right\|_{2,\Tr}<\epsilon.
\]
Clearly,
\[\left\|p-\sum_{i=1}^n x_ie_B x_i^*\right\|_{2,\Tr}<\epsilon.\]
\end{proof}

Suppose that $\H\subset L^2(M)$ is a right $B$-module. We denote by $\L_B(L^2(B),\H)$  the set of bounded right $B$-modular operators from $L^2(B)$ into $\H$. The dimension of $\H$ over $B$ is defined as
\[
 {\rm dim}_B(\H)=\Tr(1),
\]
where $\Tr$ is the unique tracial weight on $B'$ satisfying the following condition
\[
 \Tr(x^*x)=\tau(xx^*),\quad \forall x\in \L_B(L^2(B),\H).
\]
 We say $\H$ is a \emph{finite right $B$ module} if $\Tr(1)<\infty$. For details on finite modules, we refer the reader to appendix A of~\cite{Va}.

Suppose that $\H\subset L^2(M)$ is a right $B$-module. We say that $\H$ is \emph{finitely generated} if there exist finitely many elements $\xi_1,\ldots,\xi_n\in \H$ such that $\H$ is the closure of $\sum_{i=1}^n\xi_i B$. A set $\{\xi_i\}_{i=1}^n$ is called an orthonormal basis of  $\H$ if  $\E_B(\xi_i^*\xi_j)=\delta_{ij} p_i\in B$, $p_i^2=p_i$,  and for every $\xi\in \H$ we have
\[
\xi=\sum_i \xi_i E_B(\xi_i^*\xi).
\] Let $p$ be the orthogonal projection of $L^2(M)$ onto $\H$. Then $p=\sum_{i=1}^n\xi_ie_B\xi_i$, where $\xi_i\in L^2(M)$ is viewed as an unbounded operator affilated with $M$.  Every finitely generated right $B$ module has an orthonormal basis. For finitely generated right $B$ modules, we refer to 1.4.1 of~\cite{Po2}.

The following lemma is proved by Vaes in~\cite{Va} (see Lemma A.1).

\begin{Lemma}\label{L:vaes}
 Suppose $\H$ is a finite right $B$-module. Then there exists a sequence of projections $z_n$ of $Z(B)=B'\cap B$ such that $\lim_{n\rightarrow\infty}z_n=1$ in the strong operator topology and
 $\H z_n$ is unitarily equivalent to the  left $p_nM_{k_n}(B)p_n$ right $B$ bimodule  $p_n(L^2(B)^{(n)})$ for each $n$. In particular, $\H z_n$ is a finitely generated right $B$ module.
\end{Lemma}

The following lemma is motivated by Lemma 1.4.1 of~\cite{Po2}.
\begin{Lemma}\label{L:Popa}
 Suppose  $\H\subset L^2(M)$ is a left $A$ and finitely generated right $B$ bimodule. Let $p$ denote the orthogonal projection of $L^2(M)$ onto $\H$. Then there exists a sequence of projections $z_n$ in $A'\cap M$  such that $\lim_{n\rightarrow\infty}z_n=1$ in the strong operator topology and for each $n$, there exist a finite number of elements $x_{n,1},\ldots,x_{n,k}\in M$ such that
\[
 z_npz_n(\hat{x})=\sum_{i=1}^k\widehat{x_{n,i}\E_B(x_{n,i}^*x)},\quad \forall x\in M.
\]
\end{Lemma}
\begin{proof}
Let $\{\xi_i\}_{i=1}^k\subset \H\subset L^2(M,\tau)$ be an orthonormal basis for $\H$, i.e., $\H=\oplus_{i=1}^k [\xi_iB]$.  Then $p=\sum_{i=1}^k\xi_ie_B\xi_i^*\in A'\cap \langle M,e_B\rangle$, where $\xi_i\in L^2(M)$ is viewed as an unbounded operator affilated with $M$. For $a\in A$, we have
\[
 a\left(\sum_{i=1}^n\xi_ie_B\xi_i^*\right)=\left(\sum_{i=1}^n\xi_ie_B\xi_i^*\right)a.
\]
Applying the pull down map to both sides, we obtain
\[
 a\left(\sum_{i=1}^n\xi_i\xi_i^*\right)=\left(\sum_{i=1}^n\xi_i\xi_i^*\right)a.
\]
Hence $aq=qa$ for all spectral projections $q$ of $\xi_i\xi_i^*$. Since $\sum_{i=1}^n\xi_i\xi_i^*$ is a densely defined operator affilated with $M$, $q\in A'\cap M$.
  We thus obtain a sequence of projections $z_n\in A'\cap M$ such that $\lim_{n\rightarrow \infty}z_n=1$ in the strong operator topology and $\sum_{i=1}^kz_n\xi_i\xi_i^*z_n$ is a bounded operator for each $n$. Let $x_{n,i}=z_n\xi_i$, $1\leq i\leq k$. Then $x_{n,i}\in M$ and
\[
 z_npz_n(\hat{x})=\sum_{i=1}^kz_n\xi_ie_B\xi_i^*z_n(\hat{x})=\sum_{i=1}^kx_{n,i}e_Bx_{n,i}^*(\hat{x})=\sum_{i=1}^k\widehat{x_{n,i}\E_B(x_{n,i}^*x)},\quad \forall x\in M.
\]
\end{proof}

\begin{proof}[Proof of Theorem~\ref{T:main result}]
``$1\Rightarrow 2$''.
 Suppose $e_Bze_B=z$ is not true. We may assume that $(1-e_B)z\neq 0$ (otherwise, consider $z(1-e_B)$). Replacing $z$ by a nonzero spectral projection of $(1-e_B)zz^*(1-e_B)$ corresponding to an interval $[c,1]$ with $c>0$, we may assume that $z=p\neq 0$ is a subprojection of $1-e_B$.

Let $\epsilon>0$. By Lemma~\ref{L:approximate projections}, there is a natural number $n$ and
$x_1,\ldots,x_n\in M\ominus B$ such that $\E_B(x_j^*x_i)=\delta_{ij}f_i$,
where $f_i$ is a projection in $B$, and
\[\|p-\sum_{i=1}^n x_ie_B x_i^*\|_{2,\Tr}<\epsilon/2.\]
Let $p_0=\sum_{i=1}^n x_ie_B x_i^*$. Then $p_0$ is a projection. Note that $u_kpu_k^*=p$. So
\[
 \|u_kp_0u_k^*-p_0\|_{2,\Tr}\leq \|u_k(p_0-p)u_k^*\|_{2,\Tr}+\|p_0-p\|_{2,\Tr}<\epsilon.
\]
Therefore,
\[
 2\|p_0\|_{2,\Tr}^2=\|u_kp_0u_k^*-p_0\|_{2,\Tr}^2+2\Tr(u_kp_0u_k^*p_0)
\]
\[
 =\|u_kp_0u_k^*-p_0\|_{2,\Tr}^2+2\sum_{1\leq i,j\leq n} \Tr(u_kx_ie_Bx_i^*u_k^*x_je_Bx_j^*)
\]
\[
 \leq \epsilon^2+2\sum_{1\leq i,j\leq n}\tau(\E_B(x_i^*u_k^*x_j)x_j^*u_kx_i)
\]
\[
 \leq \epsilon^2+2\sum_{1\leq i,j\leq n}\|\E_B(x_j^*u_kx_i)\|_{2,\tau}^2.
\]
By the assumption of the lemma, $2\sum_{1\leq i,j\leq n}\|\E_B(x_j^*u_kx_i)\|_{2,\tau}^2\rightarrow 0$ when $k \rightarrow \infty$. Hence, $\|p_0\|_{2,\Tr}\leq \epsilon$.  Since $\epsilon>0$ was arbitrary, this says $p=0$. This is a contradiction.

 ``$2\Rightarrow 1$''. Suppose 1 is false. Then there exists an $\epsilon_0>0$ and $x_1,\ldots,x_n\in N\ominus B$ such that $\sum_{1\leq i,j\leq n}\|\E_B(x_iux_j^*)\|_{2,\tau}^2\geq \epsilon_0$ for all $u\in \U(A)$. Let $z=\sum_{i=1}^n x_i^*e_Bx_i$. Then $z\perp e_B$, $\Tr(z)<\infty$, and
\[
 \Tr(zuzu^*)=\sum_{i,j=1}^n\Tr(x_i^*e_Bx_iux_j^*e_Bx_ju^*)=\sum_{i,j=1}^n\Tr(E_B(x_iux_j^*)e_Bx_ju^*x_i^*)
\]
\[
 =\sum_{i,j=1}^n\tau(E_B(x_iux_j^*)x_ju^*x_i^*)=\sum_{i,j=1}^n\|E_B(x_iux_j^*)\|_2^2\geq \epsilon,\quad\forall u\in \U(A).
\]
Consider $\Gamma_z$, the weak operator closure of the convex hull of $\sett{uzu^*:u\in \U(A)}$. Then there exists a unique element $y\in \Gamma_z$ such that $\|y\|_{2,\Tr}=\min\{\|x\|_{2,\Tr}:\,x\in \Gamma_z\}$. The uniqueness implies that $uyu^*=y$ for all $u\in \U(A)$ and hence $y\in A'\cap \langle N,e_B\rangle$.  Since $\Tr(zuzu^*)\geq \epsilon_0$, $\Tr(zy)\geq \epsilon_0>0$. So $y>0$ and $y\perp e_B$.  Note that
\[
 \Tr(y^2)\leq \|y\|\Tr(y)\leq \|y\|\Tr(z)<\infty.
\]
This contradicts the assumption of 2.

``$2\Leftrightarrow 3$'' is easy to see.

``$3\Rightarrow 4$''. Suppose $Ax\subset \sum_{i=1}^nx_iB$. Let $\H$ be the closure of $\widehat{AxB}$ in $L^2(N,\tau)$. Then $\H$ is a left $A$  finitely generated right $B$ bimodule. Let $p$ be the projection of $L^2(N,\tau)$ onto $\H$. Then $p\in A'\cap \langle N,e_B\rangle$ is a finite projection of $\langle N,e_B\rangle$. By the assumption of 3, $p\leq e_B$. So $\hat{x}=p(\hat{x})=e_B(\hat{x})\in \hat{B}$ and $x\in B$.

``$4\Rightarrow 3$''.  Suppose $p\in A'\cap \langle M,e_B\rangle$ satisfies $\Tr(z^*z)<\infty$. Then $\H=pL^2(M)$ is a left $A$ finite right $B$ bimodule.  By Lemma~\ref{L:vaes}, we may assume that $\H$ is a left $A$  finitely generated right $B$ bimodule. By Lemma~\ref{L:Popa}, there exists a sequence of   projections $z_n$ in $A'\cap M$  such that $\lim_{n\rightarrow\infty}z_n=1$ in the strong operator topology and for each $n$, there exist $x_{n,1},\ldots,x_{n,k}\in M$ such that
\[
 z_npz_n(\hat{x})=\sum_{i=1}^k\widehat{x_{n,i}\E_B(x_{n,i}^*x)},\quad \text{for all } x\in M.
\]
  Note that $z_npz_n\in A'\cap \langle M,e_N\rangle$,
and for every $x\in M$,
\[
 A\left(z_npz_n(\hat{x})\right)=(z_npz_n)(\widehat{Ax})\subset\sum_{i=1}^n\widehat{x_{n,i}B}.
\]
 By the assumption of 4, $z_npz_n(\hat{x})\in \hat{B}\subset L^2(B)$ for every $x\in M$. Hence, for each $\xi\in L^2(M)$, $z_npz_n(\xi)\in L^2(B)$. Since $\lim_{n\rightarrow\infty}z_n=1$ in the strong operator topology, $p(\xi)=\lim_{n\rightarrow\infty} z_npz_n(\xi)\in L^2(B)$, i.e.,  $p\leq e_B$.

\end{proof}

\section{Further results and examples}  \label{section:subalgebras}
In this section, we explore the hereditary properties of mixing subalgebras of finite von Neumann algebras; that is, we show that if $B \subset M$ is a mixing inclusion, then the properties of an inclusion $B_1 \subset B$ can force certain mixing properties on the inclusion $B_1 \subset M$. In particular, Proposition \ref{P:subalgebras of mixing} below allows us to construct examples of weakly mixing subalgebras which are not mixing.
We also use the crossed product and amalgamated free product constructions to produce further examples of mixing inclusions.

\subsection{Hereditary properties of mixing algebras}
\begin{Proposition}\label{P:subalgebras of mixing} Let $B$ be a mixing von
Neumann subalgebra of $M$, and let $B_1$ be a diffuse von Neumann
subalgebra of $B$. We have the following
\begin{enumerate}
\item $B_1'\cap M=B_1'\cap B$;

\item if $B_1$ is singular in $B$, then $B_1$ is singular in $M$;

\item $\N_M(B_1)''\subseteq B$, where $\N_M(B_1)=\{u\in \U(M):\,uB_1u^*=B_1\}$;

\item if $B_1$ is weakly mixing in $B$, then $B_1$ is weakly
mixing in $M$;

\item if $B_1$ is mixing in $B$, then $B_1$ is
mixing in $M$;
\end{enumerate}
\end{Proposition}
\begin{proof}

1,2,3 follow  from Theorem~\ref{T:mixing}.

4.\, By Corollary~\ref{C:weak mixing}, we need to show that if $x\in M$ satisfies $B_1x\subset\sum_{i=1}^nx_iB_1$ for a finite number of elements $x_1,\ldots,x_n\in M$, then $x\in B_1$.  Note that $B$ is  mixing in $M$. By Corollary~\ref{C:strong mixing}, $x\in B$. Let $b_i=\E_B(x_i)$ for $1\leq i\leq n$. Applying $\E_B$ to both sides of the inclusion $B_1x\subset\sum_{i=1}^nx_iB_1$, we have $B_1x\subset\sum_{i=1}^nb_iB_1$. Since $B_1$ is weakly
 mixing in $B$, $x\in B_1$ by Corollary~\ref{C:weak mixing}.

5.\, Suppose $B_1$ is mixing in $B$ and $u_n$ is
a sequence of unitary operators in $B_1$ with
$\lim_{n\rightarrow\infty}u_n=0$ in the weak operator topology.
For $x,y\in M$, we have
\[
\lim_{n\rightarrow\infty} \|\E_B(xu_ny)-\E_B(x)u_n\E_B(y)\|_2=0
\]
since $B$ is mixing in $M$. Applying $\E_{B_1}$ to
$\E_B(xu_ny)-\E_B(x)u_n\E_B(y)$, we have
\begin{equation}\label{E:B in M}
\lim_{n\rightarrow\infty}
\|\E_{B_1}(xu_ny)-\E_{B_1}(\E_B(x)u_n\E_B(y))\|_2=0
\end{equation}
 Since $B_1$ is mixing in $B$,
\begin{equation}\label{E:B1 in B}
\lim_{n\rightarrow\infty}
\|\E_{B_1}(\E_B(x)u_n\E_B(y))-\E_{B_1}(x)u_n\E_{B_1}(u)\|_2=0.
\end{equation}
Combining (\ref{E:B in M}) and (\ref{E:B1 in B}), we have
\[\lim_{n\rightarrow\infty}
\|\E_{B_1}(xu_ny)-\E_{B_1}(x)u_n\E_{B_1}(u)\|_2=0,
\] which implies that $B_1$ is mixing in $M$.

\end{proof}

\begin{Remark}\emph{ Suppose $B_i$ is a diffuse von Neumann
subalgebra of $M_i$ for $i=1,2$. If $B_1\neq M_1$ or $B_2\neq
M_2$, then $B_1\tensor B_2$ is not a mixing von
Neumann subalgebra of $M_1\tensor M_2$ by Proposition~\ref{P:subalgebras of mixing}. On the other hand, it is easy to check that $B_1\tensor B_2$ is weakly mixing in $M_1\tensor M_2$ if $B_1$ and $B_2$ are weakly mixing in $M_1$ and $M_2$, respectively. This gives  examples of weakly mixing but not mixing subalgebras.}
\end{Remark}

Note that in the proof of statement 4 of Proposition~\ref{P:subalgebras of mixing}, we use an equivalent condition of weak mixing (Corollary~\ref{C:weak mixing}) instead of the definition. The essential difficulty is that in the definition of weak mixing, we do not assume that $\lim_{n\rightarrow\infty} u_n=0$ in the weak operator topology. However, we have the following result.

\begin{Proposition}\label{L:unitary sequences of weakly mixing}Let $M$ be a type ${\rm II}_1$ factor with the faithful normal trace $\tau$, and let
$B$ be a proper subfactor of $M$. If $\{u_n\}$ is a
sequence of unitary operators in $B$ such that for all elements
$x,y$ in $M$ with $\E_B(x)=\E_B(y)=0$, one has
\[\lim_{n\rightarrow \infty}\|\E_B(xu_ny)\|_2=0,
\] then
$\displaystyle\lim_{n\rightarrow \infty}u_n=0$ in the weak
operator topology.
\end{Proposition}
\begin{proof}Note that $B$ is weakly mixing in $M$ and hence
singular in $M$. In particular $B'\cap M=\cc 1$. Let $\omega$ be a
non principal ultrafilter of $\mathbb{N}$ and suppose
$\lim_{n\rightarrow \omega}u_n=b$ in the weak operator topology.
For $x,y$ in $M$ with $\E_B(x)=\E_B(y)=0$,
\[\E_B(xby)=\lim_{n\rightarrow\omega}\E_B(xu_ny)=0.\]
 Let
$b=u|b|$ be the polar decomposition of $b$. Note that
$\E_B(xu^*)=\E_B(x)u^*=0$. Hence,
\[\E_B(x|b|y)=\E_B(xu^*u|b|y)=\E_B(xu^*by)=0.\] Let
$x=y^*$. Then $\E_B(y^*|b|y)=0$ and hence $y^*|b|y=0$. This implies
that $|b|y=0$ for all $y\in M$ with $\E_B(y)=0$. For $b'\in B$,
$\E_B(b'y)=b'\E_B(y)=0$. Hence, $|b|b'y=0$. This implies that
$|b|R(b'y)=0$, where $R(b'y)$ is the range projection of $b'y$.
Let $p=\vee_{b'\in B}R(b'y)$. Then $|b|p=0$. On the other hand,
$0\neq p\in B'\cap M$. So $p=1$. So $|b|=0$ and $b=0$. Therefore, \
$\lim_{n\rightarrow \omega}u_n=0$ in the weak operator topology.
 Since
$\omega$ is an arbitrary non principal ultrafilter of
$\mathbb{N}$, $\lim_{n\rightarrow \infty}u_n=0$ in the weak
operator topology.
\end{proof}

\subsection{Further examples of mixing subalgebras}  \label{section:examples}

\begin{Lemma}\label{L:conditional expectation and compactness}
 Let $B$ be a von Neumann subalgebra of $M$.
Then the following conditions are equivalent:
\begin{enumerate}
\item $B$ is atomic type ${\rm I}$;
\item for every \emph{bounded} sequence $\{x_n\}$ in $M$ with
$\lim_{n\rightarrow \infty}x_n=0$ in the weak operator topology,
$\lim_{n\rightarrow \infty}\|\E_B(x_n)\|_2=0$.
\end{enumerate}
\end{Lemma}
\begin{proof}
 \noindent $1\Rightarrow
2$:\, Since $B$ is a finite atomic type ${\rm I}$ von Neumann
algebra, $B=\oplus_{k=1}^N M_{n_k}(\cc)$, where $1\leq N\leq
\infty$. So there exists a sequence of finite rank central
projections $p_n\in B$ such that $p_n\rightarrow 1$ in the strong
operator topology. Therefore, $\tau(p_n)\rightarrow 1$.  Let $\{x_n\}$'
 be a bounded sequence in $M$ with $x_n\rightarrow 0$ in the weak
operator topology, and let $\epsilon>0$. We may assume that
$\|x_n\|\leq 1$. Choose $p_k$ such that
$\tau(1-p_k)<\epsilon^2/4$. Note that the map $x\in M\rightarrow
p_k\E_B(x)$ is a finite rank operator. There is an $m>0$ such that
for all $n\geq m$, $\|p_k\E_B(x_n)\|_2<\epsilon/2$. Then
\[\|\E_B(x_n)\|_2\leq \|p_k\E_B(x_n)\|_2+\|(1-p_k)\E_B(x_n)\|_2\leq
\epsilon/2+\epsilon/2=\epsilon.
\] This proves that $\|\E_B(x_n)\|_2\rightarrow 0$.

\noindent $2\Rightarrow 1$:\, If $M$ is not atomic type ${\rm I}$.
Then there is a nonzero central projection $p\in M$ such that $p
M$ is diffuse. So there is a Haar unitary operator $v\in pM$. Note
that $v^n\rightarrow 0$ in the weak operator topology. But
$\|\E_B(v^n)\|_2=\|v^n\|_2=\tau(p)^{1/2}$ does not converge to $0$.
This contradicts to 2.
\end{proof}

\begin{Proposition}\label{P:amalgamated} Let $M=M_1*_A M_2$ be the amalgamated free
product of diffuse finite von Neumann algebras $(M_1,\tau_1)$ and
$(M_2,\tau_2)$ over an atomic finite von Neumann algebra $A$. Then
 $M_1$ is a mixing von Neumann subalgebra of $M$.
\end{Proposition}
\begin{proof} The following spaces are mutually orthogonal with
respect to the unique trace $\tau$ on $M$:
$M_2\ominus A$, $(M_1\ominus A)\otimes (M_2\ominus A)$,
$(M_2\ominus A)\otimes (M_1\ominus A)$, $(M_1\ominus A)\otimes
(M_2\ominus A)\otimes (M_1\ominus A)$, $\cdots$. Furthermore, the
trace-norm closure of the linear span of the above spaces is
$L^2(M,\tau)\ominus L^2(M_1,\tau)$. Suppose $\{u_n\}$ is a sequence
of unitary operators in $M_1$ satisfying
$\displaystyle\lim_{n\rightarrow \infty}u_n=0$ in the weak
operator topology. To prove $M_1$ is a mixing von Neumann
subalgebra of $M$, we need only to show for  $x$ in each of the
above space, we have
\[\lim_{n\rightarrow \infty}\|\E_{M_1}(xu_nx^*)\|_2=0.
\] We will give the proof for $x$ in one of the following spaces:
 $(M_1\ominus A)\otimes (M_2\ominus
A)$, $(M_2\ominus A)\otimes (M_1\ominus A)$. The other cases can
be proved similarly.

Suppose $x=x_1y_1$, where $x_1\in M_1\ominus A$ and $y_1\in
M_2\ominus A$. Then
\[xu_nx^*=x_1y_1(u_n-\E_A(u_n))y_1^*x_1+x_1y_1\E_A(u_n)y_1^*x_1^*.\]
Note that $\E_{M_1}(x_1y_1(u_n-\E_A(u_n))y_1^*x_1)=0$ and
$\lim_{n\rightarrow\infty}\|E_A(u_n)\|_2=0$ by
Lemma~\ref{L:conditional expectation and compactness}. So
\[\lim_{n\rightarrow \infty}\|\E_{M_1}(xu_nx^*)\|_2=0.
\]

Suppose $x=y_1x_1$, where $x_1\in M_1\ominus A$ and $y_1\in
M_2\ominus A$. Then
\[xu_nx^*=y_1x_1u_nx_1^*y_1=y_1(x_1u_nx_1^*-\E_A(x_1u_nx_1^*))y_1^*-y_1\E_A(x_1u_nx_1^*)y_1^*.\]
Note that $\E_{M_1}(y_1(x_1u_nx_1^*-\E_A(x_1u_nx_1^*))y_1^*)=0$ and
$\lim_{n\rightarrow\infty}\|\E_A(x_1u_nx_1^*)\|_2=0$ by
Lemma~\ref{L:conditional expectation and compactness}. So
\[\lim_{n\rightarrow \infty}\|\E_{M_1}(xu_nx^*)\|_2=0.
\]
\end{proof}

Note, in particular, that Proposition \ref{P:amalgamated} implies that if $A$ is a diffuse mixing masa in a finite von Neumann algebra $M_1$, and $M_2$ is also diffuse, then $A$ is mixing in the free product $M_1 \ast M_2$.

Now let $B$ be a diffuse finite von Neumann algebra with a faithful normal trace $\tau$, and let $G$ be a countable discrete group. Let $\ast_{g\in G}B_g$ be the free product von Neumann algebra, where $B_g$ is a copy of $B$ for each $g$.   The shift transformation $\sigma(g)((x_h))=(x_{g^{-1}h})$ defines an action of $G$ on $\ast_{g\in G}B_g$. Let $M=\ast_{g\in G}B_g\rtimes G$. Then $M$ is a type ${\rm II}_1$ factor and we can identify $B$ with $B_e$.

\begin{Proposition}
 The above algebra $B$ is a mixing von Neumann subalgebra of $M$.
\end{Proposition}
\begin{proof}
 Suppose $v_g$ is the classical unitary operator corresponding to the action $g$ in $M$. Then for every $(x_h)$ in $\ast_{g\in G}B_g$,
\[
 v_g(x_h)v_g^{-1}=(\sigma_g(x_h))=(x_{g^{-1}h}).
\]
 Suppose $b_n\in B=B_e$, $b_n\rightarrow 0$ in the weak operator topology,  $g\neq e$, and $x_h\in B_h$. We may assume $\tau(b_n)=0$ for each $n$. Note that
\[
 x_hv_gv_nv_g^*x_h^*=x_h\sigma_g(b_n)x_h^*.
\]
 If $h\neq e$, it is clear that $x_h\sigma_g(b_n)x_h^*$ is free with $B=B_e$ and hence orthogonal to $B$.  If $h=e$, direct computations show that $x_e\sigma_g(b_n)x_e^*$ is orthogonal to $B=B_e$.  So we have
\[
\E_{B}(x_hv_gb_nv_g^*x_h^*)=\E_{B}(x_h\sigma_g(b_n)x_h^*)=\tau(x_h\sigma_g(b_n)x_h^*)=\tau(\sigma_g(b_n)x_h^*x_h)=\tau(b_n\sigma_{g^{-1}}(x_h^*x_h)),
\]
and this last expression above converges to zero.  Note that the linear span of the  above elements $x_hv_g$ is dense in $M\minus B$ in the weak operator topology. This proves that $B$ is mixing in $M$.
\end{proof}


\end{document}